\documentclass[11pt,reqno]{amsart}
\usepackage{amsmath,amsthm,amssymb}
\usepackage{amssymb}

\usepackage{graphics,graphicx}
\usepackage{hyperref}
\usepackage[usenames, dvipsnames]{xcolor}

\definecolor{darkblue}{rgb}{0.0,0.0,0.3}
\hypersetup{colorlinks,breaklinks,
  linkcolor=darkblue,urlcolor=darkblue,
anchorcolor=darkblue,citecolor=darkblue}

\usepackage[square,sort,comma,numbers]{natbib}
\usepackage{fancyvrb}
\usepackage{enumitem}

\theoremstyle{plain}
\newtheorem{theorem}{Theorem}
\newtheorem*{theorem*}{Theorem}
\newtheorem{lemma}[theorem]{Lemma}

\newtheorem*{proposition*}{Proposition}
\newtheorem{corollary}[theorem]{Corollary}
\newtheorem*{corollary*}{Corollary}

\theoremstyle{definition}
\newtheorem{remark}[theorem]{Remark}

\DeclareMathOperator*{\Vol}{Vol}

\title{A Shifted Sum for the Congruent Number Problem}
\author[Hulse]{Thomas A. Hulse}
\author[Kuan]{Chan Ieong Kuan}
\author[Lowry-Duda]{David Lowry-Duda}
\author[Walker]{Alexander Walker}

\begin{document}

\begin{abstract}
  We introduce a shifted convolution sum that is parametrized by the squarefree
  natural number $t$.
  The asymptotic growth of this series depends explicitly on whether or not $t$ is a
  \emph{congruent number}, an integer that is the area of a rational right triangle. This series presents a new avenue of inquiry for The Congruent Number Problem. 
\end{abstract}

\maketitle

A \emph{congruent} number is an integer which appears as the area of a right triangle
with rational-length sides.
The Congruent Number Problem is the classification problem of determining which integers are congruent numbers.
%is concerned with finding a terminating algorithm that can determine if $t\in\mathbb{Z}$ is a congruent number in an amount of
%time dependent only on the size of $t$.
For example, it is known that $5$, $6$ and $7$ are congruent numbers but that
$1$, $2$, $3$, and $4$ are not.
(For more background on the Congruent Number Problem, see the excellent survey
written by Conrad~\cite{conrad}).

The observation that scaling the side lengths of a triangle by $\lambda$ scales its area by $\lambda^2$ reduces the Congruent Number Problem to that of determining which \emph{squarefree} integers are congruent.  Restricting to this case, we see by another rescaling that $t$ (squarefree) is congruent if and only it appears as the squarefree part of the area of a right triangle with \emph{integral sides}.
%By scaling the sides of rational right triangles by rational multiples, we can restrict attention to those $t$ which are squarefree without loss of generality.
%It follows that $t$ is a congruent number if and only if it is the squarefree part of the area of a right triangle with integer-length sides.

The Pythagorean identity, $a^2 + b^2 = c^2$, shows that common divisors of
any two sides of a right triangle must also divide the third side, and
thus we can further assume that such an integer-sided triangle is
\emph{primitive}, i.e.\ that the side lengths are pairwise coprime.
We also note the classical observation that the area, $ab/2$, of a primitive right
triangle is an integer, as one of the legs must have even length.

The main object is this paper is the following theorem, which relates a shifted
sum of arithmetic functions to the hypotenuses of primitive right triangles
with squarefree part of the area $t$.

\begin{theorem}\label{theorem:one}
Let $t \in \mathbb{N}$ be squarefree, and let $\tau: \mathbb{Z} \to \{0,1\}$,
where $\tau(n)=1$ if $n$ is a square and $\tau(n)=0$ otherwise.
Let $r$ be the rank of the elliptic curve $E_t: y^2 = x^3-t^2x$ over
$\mathbb{Q}$.
For $X > 1$, define the shifted partial sum
\begin{equation*}
  S_t(X): = \sum_{m=1}^{X} \sum_{n=1}^{X} \tau(m+n) \tau(m-n) \tau(m) \tau(tn).
\end{equation*}
Then this sum has the asymptotic expansion
\begin{equation*}
  S_t(X) =  C_t X^{\frac{1}{2}} +O_t((\log X)^{r/2}),
\end{equation*}
in which
\begin{equation*}
  C_t := \sum_{h_i \in \mathcal{H}(t)} \frac{1}{h_i}
\end{equation*}
is the convergent sum over $\mathcal{H}(t)$, the set of hypotenuses, $h_i$, of
dissimilar primitive right triangles with squarefree part of the area $t$.
\end{theorem}

Note that $C_t \neq 0$ if and only if $t$ is a squarefree congruent number.
The sum $S_t(X)$ detects 3-term arithmetic progressions of squares whose common
difference has squarefree part $t$.
There is a well-known one-to-one correspondence between 3-term arithmetic
progressions of squares with common difference $t$ and right triangles with area
$t$, given by
\begin{align*}
  \big\{
    (\alpha, \beta, \gamma): \beta^2 - \alpha^2 = t = \gamma^2 - \beta^2
  \big\}
  \leftrightarrow&
  \big\{ (a,b,c): a^2 + b^2 = c^2, \; ab/2 = t \big\}
  \\
  (\alpha, \beta, \gamma) \mapsto (\gamma-\alpha, \gamma + \alpha, 2 \beta),
  \quad&
  (a,b,c) \mapsto \big( \tfrac{b-a}{2}, \tfrac{c}{2}, \tfrac{b+a}{2} \big).
\end{align*}
We adopt the convention that $a<b<c$ for a Pythagorean triple $(a,b,c)$. 
Using the above correspondence, it's straightforward to see that $S_t(X)$ is nonzero
for sufficiently large $X$ if and only if $t$ is congruent.
But the understanding of the main term and separation from the error term in
Theorem~\ref{theorem:one} are new.

Relating the Congruent Number Problem to the sum $S_t(X)$ is of specific
interest to the authors and their ongoing investigation of $S_t(X)$ (and
analogous shifted sums) via spectral expansions arising from shifted multiple
Dirichlet series associated to modular forms.
Adapting the recent methods for studying double shifted sums
from~\cite{hkldw1, HKLW5} and for studying triple shifted sums in~\cite{Tom},
it should be possible to directly study sums similar to $S_t(X)$ through
spectral methods of automorphic forms.
Such a path seems an unstudied approach to the Congruent Number Problem.

\section{Proof of Main Theorem}

The remainder of this paper is dedicated to proving Theorem~\ref{theorem:one}.
The overall idea is to first study the bijection between primitive Pythagorean
triples and 3-term arithmetic progressions of squares, and relate $S_t(X)$ to
sums of reciprocals hypotenuses of right triangles through this bijection.
We then use a bijection between primitive Pythagorean triples and points on
elliptic curves, and show how to control the error term in
Theorem~\ref{theorem:one} through estimates of height functions on these
elliptic curves.

\subsection*{Primitive Pythagorean triples and progressions of squares.}
We begin with three lemmata that establish the connection between Pythagorean
triples and arithmetic progressions of squares.

\begin{lemma}\label{lem:bijection}
  The set of primitive Pythagorean triples $(a,b,c)$, where $a<b<c$, with squarefree part of the area
  equal to $t$, and the set of positive coprime pairs $(m, n)$, for which $m^2 + tn^2$ and
  $m^2 - tn^2$ are both squares, are in bijection. The bijective maps are
  \begin{align}
    (a,b,c) &\mapsto (c, \sqrt{2ab/t})=(m,n), \label{hypo} \\
    (m,n) &\mapsto \left(
      \tfrac{\sqrt{m^2+tn^2}-\sqrt{m^2-tn^2}}{2}
      \cdot
      \tfrac{\sqrt{m^2+tn^2}+\sqrt{m^2-tn^2}}{2}, m
    \right) = (a,b,c). \label{line:reverse_map}
  \end{align}
\end{lemma}

\begin{proof}
It is straightforward to show that the maps in~\eqref{hypo}
and~\eqref{line:reverse_map} are inverses to each other, so it remains only to
check that the functions take values in the indicated sets.

Let $(a,b,c) \in \mathbb{N}^3$ be a primitive Pythagorean triple with
squarefree part of the area $t$.
Considering $a^2 + b^2 \equiv c^2 \bmod 4$, we see that $c$ must be odd.
Let $(m,n)$ denote the image of $(a,b,c)$ through the map in~\eqref{hypo} and
note that $m^2-tn^2=(a-b)^2$, $m^2 +tn^2=(a+b)^2$, and that $m$ is odd.
We see also that $\gcd(m,n) = 1$, for if $p$ is a common odd prime divisor
of $m=c$ and $n = \sqrt{2ab/t}$ then $p\mid a$ or $p\mid b$, contradicting the
primitivity of $(a,b,c)$.

Conversely, fix a squarefree integer $t$ and consider coprime $(m,n) \in
\mathbb{N}^2$ such that  $m^2-tn^2$ and $m^2+tn^2$ are both squares, and let
$(a,b,c)$ denote the image of $(m,n)$ through the map
in~\eqref{line:reverse_map}.
One quickly verifies that $a^2 + b^2 = c^2$ and $ab/2 = t (n/2)^2$, so that
$(a,b,c)$ is (at least) a rational right triangle with an area with squarefree part $t$.
Clearly $(2a, 2b, 2c)$ is an integral right triangle with even
hypotenuse $2c$, and considering $(2a)^2 + (2b)^2 \equiv (2c)^2 \bmod 4$, one
sees that both $2a$ and $2b$ must be even integers.
It follows that $(a,b,c)$ is integral, and it remains only to show that it is
primitive.
If not, let $p$ be a common divisor of $a$ and $b$; then $p^2$ divides both $a^2 + b^2
= m^2$ and $2ab = tn^2$, so $p \mid \gcd(m,n) = 1$, a contradiction.
\end{proof}

%\davidnote{This was a sentence at the end of the previous lemma. I don't know
%why we care about this, though it is mildly (?) interesting.}
%\begin{remark}
%  As a corollary to this lemma, we see that if $m^2 - tn^2, m^2, m^2 + tn^2$ is
%  an arithmetic progression of three squares where $\gcd(m,n) = 1$ and $t$ is
%  squarefree, then $m$ is odd.
%  One can also prove this directly by considering $m^2 - tn^2, m^2, m^2
%  + tn^2 \bmod 4$.
%\end{remark}

We are now ready to relate $S_t(X)$ to a sum of inverse hypotenuse lengths.

\begin{lemma}\label{lem3}
  With notation as in Theorem~\ref{theorem:one}, we have
  \begin{equation*}
    S_t(X)
    =
    \sum_{\substack{h_i \leq X^{\frac{1}{2}}\\ h_i \in \mathcal{H}(t)}}
    \left\lfloor \frac{X^{\frac{1}{2}}}{h_i} \right\rfloor,
  \end{equation*}
  % NOTE: there is no good substack alignment, is there?
  % This is a problem which we shouldn't fix, but which is very unpleasant to
  % the eye. I'm mildly offended at the unaligned $h_i$ piece.
  where $\lfloor \cdot \rfloor$ is the floor function.
\end{lemma}

\begin{proof}[Proof of lemma.]
Consider the sum
\begin{equation*}
  s_t(X):=\sum_{m = 1}^X \sum_{n = 1}^X \tau(m^2-tn^2) \tau(m^2+tn^2)
\end{equation*}
By Lemma~\ref{lem:bijection}, each coprime pair $(m,n)$ contributes a nonzero
term to the sum if and only if $m$ is the hypotenuse of a primitive right
triangle with an area with squarefree part $t$.
We note that there is potential multiplicity for triangles that have the same
hypotenuse but different area with the same squarefree part, and in such a case
these triangles would necessarily be dissimilar. 
Further, the coprime pair $(m,n)$ contributes to the sum if and only if
$(rm,rn)$ also contributes to the sum for all $r \in \mathbb{N}$.
Identifying $(m,n)$ with the hypotenuse $m=h_i$, and noting that $m > n$ in any
contributing pair $(m,n)$, we can rewrite the sum as
\begin{equation}\label{lemmasum}
  s_t(X)
  :=
  \sum_{\substack{h_i \leq X \\  h_i \in \mathcal{H}(t)}}
  \sum_{r \leq X/h_i} 1
  =
  \sum_{\substack{h_i \leq X \\  h_i \in \mathcal{H}(t)}}
  \left\lfloor \frac{X}{h_i} \right\rfloor.
\end{equation}
On the other hand, as $\tau(n) = 0$ for $n < 0$,
\begin{align}
 s_t(X)
 &=
 \sum_{m =1}^X \sum_{n =1}^X \tau(m^2-tn^2)\tau(m^2+tn^2)
 \\
 &=
 \sum_{m=1}^{X^2} \sum_{n=1}^{X^2} \tau(m-tn)\tau(m+tn)\tau(m)\tau(n) \notag
 \\
 &=
 \sum_{m = 1}^{X^2} \sum_{n=1}^{X^2}\tau(m-n)\tau(m+n)\tau(m)\tau(tn), \notag
\end{align}
so $s_t(X^{\frac{1}{2}}) = S_t(X)$.
With~\eqref{lemmasum}, this completes the proof of the lemma.
\end{proof}

\begin{lemma}\label{lem4}
  The integer $m$ is the hypotenuse of a primitive right triangle $T$ with
  squarefree part of the area $t$ if and only if $m$ is the numerator (after
  reducing the fraction) of the hypotenuse of the rational right triangle
  with area $t$ which is similar to $T$.
\end{lemma}

\begin{proof}[Proof of lemma.]
Recall that the area of a right triangle with integer sides is an integer, and
suppose that $m$ is the hypotenuse of a primitive right triangle $(a,b,m)$ with
area $ab/2=tn^2$ for some $n \in \mathbb{N}$.
Then $(\frac{a}{n},\frac{b}{n},\frac{m}{n})$ is a similar rational right
triangle with area $t$.
Further, $\gcd(m,n)=1$ since any prime factor of $n$ is a factor of $a$ or $b$,
and by primitivity $\gcd(a,m)=\gcd(b,m)=1$.

Conversely, let $\frac{m}{n}$ be the hypotenuse (with reduced terms) of a
rational right triangle $(R_1,R_2,\frac{m}{n})$ with area $R_1R_2/2=t$.
After scaling by a rational $N$, we find a similar primitive right triangle
$(a,b,m')$ with area $tN^2$.
Recalling again that the area of an integer-sided right triangle is an integer
and that $t$ is squarefree, we have that $N \in \mathbb{N}$ and, by the first part of this proof,
$(m',N)=1$. Since $\frac{m}{n}N=m'$, we have $\frac{m}{n}=\frac{m'}{N}$, and both fractions are reduced so $m=m'$. \end{proof}

\begin{remark}
  As a consequence of Lemmas~\ref{lem3} and~\ref{lem4}, we see that
  the correspondence in~\eqref{hypo} taking $(a,b,c) \to (m,n)$ is also
  the map that sends a primitive right triangle $(a,b,c)$ with squarefree part
  of the area $t$ to the coprime pair $(m,n)$, such that $m$ is the numerator
  and $n/2$ is the denominator of the hypotenuse of the similar rational
  right triangle with area $t$.
\end{remark}

\subsection*{Hypotenuses and Elliptic Curves}

%(see for example~\cite[Theorem~4.1]{conrad})
There is a well-known one-to-one correspondence between rational triples
$(a,b,c)$ with $a^2+b^2=c^2$ and $ab/2 =t$, where $t$ is squarefree, and
$\mathbb{Q}$-rational points on the elliptic curve $E_t: \, y^2=x^3-t^2x$ where
$y \neq 0$, with maps
\begin{equation}\label{corr}
  \begin{split}
    (a,b,c)
    &\mapsto
    \left(\frac{tb}{c-a}, \frac{2t^2}{c-a} \right)=(x,y), \\
    (x,y)
    &\mapsto
    \left(\frac{x^2 - t^2}{y}, \frac{2tx}{y}, \frac{x^2 + t^2}{y}\right).
  \end{split}
\end{equation}
One can verify this correspondence through direct computation, or refer
to~\cite[\S4]{conrad} for further description and exposition.
We note that this family of triples will, up to change of order and sign, count
each rational right triangle with area $t$ exactly eight times:
\begin{equation}\label{alts}
  \begin{split}
    (a,b,c), \ (-a,-b,c), \ (a,b,-c), \ (-a,-b,-c), \\
    (b,a,c), \ (-b,-a,c), \ (b,a,-c), \ (-b,-a,-c).
 \end{split}
\end{equation}

The torsion subgroup of $E_t(\mathbb{Q})$ is isomorphic to
$\mathbb{Z}/2\mathbb{Z} \times \mathbb{Z}/2\mathbb{Z}$~\cite[Lemma~4.20]{knapp},
and the nontrivial torsion points are given explicitly by $T_1 = (0,0)$,
$T_2=(t,0)$ and $T_3=(-t,0)$.
These are precisely those points $(x,y)$ on $E_t(\mathbb{Q})$ with $y = 0$.
Every other point has infinite order.
It is also known that if the triple $(a,b,c)$ corresponds to the point $P$ on
$E_t(\mathbb{Q})$ then each of the alternate forms in~\eqref{alts} uniquely
corresponds to one of the following (though not necessarily respectively):
\begin{equation}\label{alts2}
  P, \  P+T_1, \ P+T_2, \ P+T_3, \ -P, \ -P+T_1, \ -P+T_2, \ -P+T_3.
\end{equation}

Let $(a,b,c)$ be one of the rational triples described above, so that
$(\lvert a \rvert,\lvert b \rvert,\lvert c \rvert)$ describes a rational right
triangle with area $t$ and hypotenuse $\lvert c \rvert$.
Let $d$ denote the denominator of $\lvert c \rvert$.
From Lemma~\ref{lem4}, $(\lvert a \rvert d,\lvert b \rvert d,\lvert c \rvert d)$
is a primitive Pythagorean right triangle with area $td^2$ and hypotenuse
$|c|d$.
The images of these two points under the map in~\eqref{corr}
agree and are given by
\begin{equation}\label{norm}
  \left(\frac{tb}{c-a}, \frac{2t^2}{c-a} \right)
  =
  \left(\frac{tbd}{cd-ad}, \frac{2t^2d}{cd-ad} \right).
\end{equation}

For $(x,y) \in E_t(\mathbb{Q})$ we define the height function,
\begin{equation*}
  H(x,y) := \max\{\lvert u \rvert ,\lvert v \rvert \}
\end{equation*}
where $x=\frac{u}{v}$ is the reduced fraction of $x$.
Then~\eqref{norm} implies that
\begin{equation}\label{lowbound}
  H\left(\frac{tb}{c-a}, \frac{2t^2}{c-a} \right)
  \leq \max
  \big\{
    t\lvert b \rvert d,\lvert c-a \rvert d
  \big\}
  \leq 2t\lvert c \rvert d.
\end{equation}

Let $h_i \in \mathcal{H}(t)$ correspond to the rational right triangle
$(a,b,c)$, and therefore to \emph{any} of the points $P_{i,j}$ on
$E_t(\mathbb{Q})$ where $j \in \{1,\ldots,8\}$ that correspond to each of the
alternate forms listed in~\eqref{alts2}.
Then $h_i = \lvert c \rvert d$ and by~\eqref{lowbound} we have that
\begin{equation*}
  h_i \geq \frac{1}{2t} H(P_{i,j}),
\end{equation*}
and in particular if $h_i \leq X/2t$ then necessarily $H(P_{i,j}) \leq X$.
Combining this with the one-to-one correspondence in~\eqref{corr} and the $8$-fold
multiplicity noted in~\eqref{alts} and~\eqref{alts2}, we see that
\begin{equation}\label{heightbound}
  \big\lvert\{ P \in E_t(\mathbb{Q}) : H(P) \leq X\}\big\rvert
  \geq
  8 \big\lvert
    \left\{ h_i \in \mathcal{H}(t) : h_i \leq X/2t \right\}
  \big\rvert.
\end{equation}

In his 1965 Annals paper, Andr\'{e} N\'{e}ron presented a theory \cite{neron} of the canonical height function in order to count the number of rational points of bounded height on an elliptic curve. Here we use a version of that result presented as Proposition~4.18 in Anthony Knapp's book  \cite{knapp} on elliptic curves.
\begin{theorem}
  Let $E$ be the elliptic curve $E: x^3 + Ax + B$ with $A,B \in \mathbb{Z}$.
  Suppose $E$ has rank $r$, and let $H(\cdot)$ be the height function defined
  above.
  Then as $X \to \infty$,
  \begin{equation*}
    \left|\{ P \in E(\mathbb{Q}) : H(P) \leq X\}\right|
    \begin{cases}
      = 
       |E(\mathbb{Q})_{\mathrm{tors}}| & \text{if } r = 0, \\
      \sim \dfrac{|E(\mathbb{Q})_{\mathrm{tors}}| \Vol(B_r)}
                 {R_{E/\mathbb{Q}}^{1/2}}
           (\log X)^{r/2}
           & \text{if } r \geq 1,
    \end{cases}
  \end{equation*}
  where $\sim$ means ``asymptotic to'',  $|E(\mathbb{Q})_{\mathrm{tors}}|$ is the
  size of the torsion subgroup, $\Vol(B_r)$ is the volume of the $r$-dimensional
  unit ball, and $R_{E/\mathbb{Q}}$ is the regulator of $E$ over $\mathbb{Q}$.

  In particular, for $E_t$, we have the simple bound
  \begin{equation*}
    |\{ P \in E(\mathbb{Q}) : H(P) \leq X\}|
    =
    O_t\big( (\log X)^{r/2} \big).
  \end{equation*}
\end{theorem}

From this theorem and~\eqref{heightbound}, we get the following corollary
bounding the number of hypotenuses.
\begin{corollary}\label{cor1}
  Let $\mathcal{H}(t)$ denote the set of hypotenuses of dissimilar Pythagorean
  right triangles with squarefree part of the area $t$.
  Then
  \begin{equation*}
    \big\lvert
      \left\{ h_i \in \mathcal{H}(t) : h_i \leq X/2t \right\}
    \big\rvert
    =
    O_t((\log X)^{r/2}),
  \end{equation*}
  where $r$ is the rank of the elliptic curve $E_t(\mathbb{Q})$.
\end{corollary}

We are now ready to complete the proof of the main theorem.
To simplify notation, take all sums below to be over $h_i \in \mathcal{H}(t)$.
From the corollary, we have that for $n \geq 0$,
\begin{equation*}
  \sum_{1 < \frac{h_i}{2^n X} \leq 2} \frac{1}{h_i}
  =
  O_t\left(\frac{(\log 2^{n+1}X)^{r/2}}{2^n X}\right).
\end{equation*}
Summing dyadically, we find that
\begin{equation}\label{diadic}
  \sum_{h_i > X} \frac{1}{h_i}
  =
  \sum_{n=0}^\infty \sum_{1 < \frac{h_i}{2^n X} \leq 2} \frac{1}{h_i}
  =
  O_t\left(\frac{(\log X)^{r/2}}{X}\right).
\end{equation}
Thus $C_t := \sum_{h_i} \frac{1}{h_i}$ is indeed convergent.

Beginning with Lemma~\ref{lem3}, we have
\begin{equation*}
  S_t(X)
  =
  \sum_{h_i \leq X^{\frac{1}{2}}}
  \left\lfloor \frac{X^{\frac{1}{2}}}{h_i} \right\rfloor
  =
  \sum_{h_i \leq X^{\frac{1}{2}}}
  \left(\frac{X^{\frac{1}{2}}}{h_i} + O(1) \right).
\end{equation*}
By Corollary~\ref{cor1}, the $O(1)$ error terms contribute no more than
$O_t((\log X)^{r/2})$, and by~\eqref{diadic} the sum over inverse hypotenuse
lengths is $C_t$ minus a rapidly convergent sum.
In particular,
\begin{align*}
  S_t(X)
  &=
  X^{\frac{1}{2}} \left(\sum_{h_i \leq X^{\frac{1}{2}}} \frac{1}{h_i}\right)
  +
  O_t((\log X)^{r/2})
  \\
  &= X^{\frac{1}{2}}\left( C_t-\sum_{h_i>X^{\frac{1}{2}}} \frac{1}{h_i}\right)
  +
  O_t((\log X)^{r/2})
  \\
  &=
  X^{\frac{1}{2}}
  \left( C_t-O_t\left(\frac{(\log X)^{r/2}}{X^{\frac{1}{2}}}\right)\right)
  +
  O_t((\log X)^{r/2})
  \\
  &=
  C_t X^{\frac{1}{2}} + O_t((\log X)^{r/2}).
\end{align*}
This completes the proof of Theorem~\ref{theorem:one}.

\section*{Acknowledgments}
The third author gratefully acknowleges support from EPSRC Programme Grant
EP/K034383/1 LMF:\ L-Functions and Modular Forms.

The authors would like to thank Asamoah Nkwanta of Morgan State University for a
conversation that inspired the authors to run explicit computations of $S_t(X)$
in SageMath which, in turn, led to the discovery of the main theorem of this
paper. The authors would also like to thank Joseph Silverman of Brown University for providing the authors with a more accurate attribution for Theorem 6. 

\bibliography{jobbib}
\bibliographystyle{abbrv}
\end{document}